\documentclass{article}
\usepackage{amsmath}
\usepackage{amsthm}
\usepackage{graphicx}
\usepackage{amsfonts}
\usepackage{amssymb}
\usepackage{color}
\theoremstyle{theorem}
\newtheorem{theorem}{Theorem}

\theoremstyle{corollary}

\theoremstyle{definition}

\def\qed{\hfill $\square$}

\begin{document}

\title{A combinatorial identity with applications to forest graphs.}
\author{Tony Dorlas\thanks{Dublin Institute for Advanced Studies,
School of Theoretical Physics, Dublin, Ireland},\and Alexei Rebenko\thanks{Institute of Mathematics, Ukrainian National Academy of Sciences, Kyiv, Ukraine}, \and Baptiste Savoie\footnotemark[1]\,\,$^{,}$\thanks{Corresponding author - e-mail: baptiste.savoie@gmail.com}}
\date{}

\maketitle

\begin{abstract}
\noindent We give an elementary proof of an interesting combinatorial identity which is of particular interest in graph theory and its applications. Two applications to enumeration of forests with closed-form expressions are given.
\end{abstract}
\vspace{0.5cm}

\noindent The aim of this paper is to give a new, elementary proof of the combinatorial identity \eqref{fdeq1} in Theorem \ref{thm} and give some non-trivial applications to enumeration of forests. Done by induction, the proof we give is simple in that it only requires the use of the binomial formula along with the derivation operator. This identity finds an interest in graph theory (for enumeration of forests) and its applications. For instance, we came across \eqref{fdeq1} when evaluating, within the framework of rigorous statistical mechanics, contributions from forest graphs to a cluster expansion for classical gas correlation functions in \cite{DRS}. Other applications of such rooted graphs in theoretical physics are for example in the works \cite{GalNic1,GalNic2,DunIagS,DunS}. The combinatorial identity \eqref{fdeq1} provides a means to directly derive a closed-form expression for the number of distinct rooted forests on a fixed collection of sets of vertices, see formula \eqref{fdeq2} in Theorem \ref{thm2}.  A more complicated situation involving additional vertices with different weighting factors is also treated, see formula \eqref{Qnm} in Theorem \ref{thm3}. The identity \eqref{fdeq1} is in fact known: see \cite{Stan}, Theorem 5.3.4, Equation (5.47). However, it is formulated and proved there in terms of forest graphs, and thereby not easily recognizable. Here we present a direct and self-contained proof. We also note that formula \eqref{fdeq1} is similar to 
Hurwitz' generalization of Abel's formula, see, e.g., \cite{Abel,Hurwitz,Pitman}.

\section{The combinatorial identity.}

\begin{theorem}
\label{thm}
Given $m,p \in \mathbb{N}$ with $1\leq p \leq m$, and a collection $(x_{i})_{i=1}^{m}$ of
(complex) numbers $x_{i} \in \mathbb{C}$, the following identity holds
\begin{equation}
\label{fdeq1}
\sum_{\{I_{1},\dots,I_{p}\} \in \Pi_p(\{1,\dots,m\})} \prod_{j=1}^{p} \left(\sum_{i \in I_{j}}
x_i\right)^{\vert I_{j} \vert -1} = {m-1 \choose p-1} \left(\sum_{i=1}^{m} x_i \right)^{m-p},
\end{equation}
where the sum on the left-hand side is over the set $\Pi_p(\{1, \dots,m\})$ of all partitions $(I_{j})_{j=1}^{p}$ of $\{1,\dots,m\}$ into $p$ non-empty subsets.
\end{theorem}

The proof of Theorem \ref{thm} is postponed to Sec. \ref{app}.

\section{Two applications to enumeration of forests.}
\label{app}

A \textit{rooted (or directed) forest} is defined on a collection $(V_{i})_{i=1}^{m}$ of sets of vertices $V_{i}$ as a graph on
$V = V_{1} \cup \cdots \cup V_{m}$ with connected components which are \textit{trees}, such that there are no lines between vertices of any individual $V_{j}$, and moreover, if each set $V_j$ is reduced to a single point and lines eminating from it to a single line, then the graph reduces to a single (connected) tree. Further, if in the reduced tree, $V_j$ ($j \geq 2$) is connected to $V_i$ in the path from $V_j$ to $V_1$, then all lines between $V_i$ and $V_j$ eminate from a single vertex in $V_i$ (a \textit{root}). This definition may appear complicated, but in fact in many situations such directed forests occur naturally. See, e.g., \cite{GalNic1,GalNic2,DunIagS, DunS,DRS,Pitman}. Deriving a closed-form expression for the number of (distinct) forests as described above becomes straightforward by using Theorem \ref{thm}. Note that a direct proof of Theorem~\ref{thm2} was given in \cite{DRS}.

\begin{theorem}
\label{thm2}
Given $m \in \mathbb{N}$, the number $N_{m}$ of distinct forests on a collection of sets of vertices $(V_{i})_{i=1}^{m}$  is given by
\begin{equation}
\label{fdeq2}
N_{m} = \vert V_{1}\vert \left(\sum_{i=1}^{m} \vert V_{i}\vert\right)^{m-2} \prod_{i=2}^{m} \left(2^{\vert V_{i} \vert}-1\right).
\end{equation}
\end{theorem}

\begin{proof} Set
\begin{equation*}
\rho_{i} := 2^{\vert V_{i}\vert} - 1,\quad i=1,\dots,m.
\end{equation*}
For $m=2$ we can choose the root in $\vert V_{1}\vert$ ways and connect it to $k=1,\dots,\vert V_{2}\vert$ points in $V_{2}$. This results in $\vert V_{1}\vert \rho_2$ different forests (in fact, disregarding isolated vertices, they are trees). Suppose now that the statement is true for $m$. Considering a forest on $m+1$ sets of vertices, the vertices of $V_{m+1}$ can be connected to $p=1,\dots,m$ other sets $V_{j}$ in the reduced tree on $V_{1},\dots,V_{m}$. Omitting the connections to $V_{m+1}$ we obtain $p$ separate forests given by subsets $I_{1},\dots,I_{p}$ of $\{1,\dots,m\}$. If $I_{1}$ contains $1$, then this yields a factor
\begin{equation*}
\vert V_{1}\vert \left(\prod_{i \in I_1\setminus\{1\}} \rho_i \right) \left(\sum_{i \in I_{1}} \vert V_{i}\vert\right)^{\vert I_{1}\vert -2} \left(\sum_{i \in I_{1}} \vert V_{i}\vert\right)\rho_{m+1} ,
\end{equation*}
the additional factor $\left(\sum_{i \in I_{1}} \vert V_{i} \vert\right)$ being due to the choice of root for the connection to $V_{m+1}$. Similarly, the other branches yield factors
\begin{equation*}
\vert V_{m+1}\vert \left(\sum_{i \in I_r} \vert V_{i} \vert\right)^{\vert I_{r}\vert-1}  \prod_{i \in I_{r}} \rho_i, \quad r=2,\dots,m.
\end{equation*}
The resulting expression for $N_{m+1}$ is
\begin{equation*}
\vert V_{1}\vert \sum_{p=1}^m \vert V_{m+1}\vert^{p-1} \sum_{\{I_1,\dots,I_p\} \in \Pi_{p}(\{1,\dots,m\})} \prod_{r=1}^p
\left(\left(\sum_{i \in I_{r}} \vert V_{i} \vert\right)^{\vert I_r\vert-1} \prod_{\substack{i \in I_r \\ i \neq 1}}\rho_i\right)\rho_{m+1}.
\end{equation*}
Identity \eqref{fdeq1} yields
\begin{equation*}
\begin{split}
N_{m+1}&= \vert V_{1}\vert \left(\sum_{p=1}^{m}  {m-1 \choose p-1} \vert V_{m+1}\vert^{p-1}
\left(\sum_{i=1}^m \vert V_{i}\vert \right)^{m-p}\right) \prod_{i=2}^{m+1} \rho_i \\
&= \vert V_{1}\vert \left(\sum_{i=1}^{m+1} \vert V_{i}\vert\right)^{m-1} \prod_{i=2}^{m+1} \rho_i.
\end{split}
\end{equation*}
This concludes the proof of Theorem \ref{thm2}.
\end{proof}

Another, more complicated application is considered in Theorem \ref{thm3} below. We refer the readers to Figure \ref{fig1} in Sec. \ref{grph} for an illustration of a generic configuration.

\begin{theorem}
\label{thm3}
Given $m, n \in \mathbb{N}$, let $(V_i)_{i=1}^m$ be a collection of sets of vertices and let $v_1,\dots,v_n$ be $n$ additional single vertices. Define a quantity $Q_n(m)$ as follows. Every forest on $\bigcup_{i=1}^m V_i \cup\bigcup_{j=1}^n \{v_j\}$ gives a contribution to $Q_n(m)$ defined by positive constants $a$ and $\lambda$ such that every edge between two additional vertices $v_k$ and $v_l$ contributes a factor $a$ whereas all other edges contribute a factor $\lambda$. Then
\begin{equation}
\label{Qnm}
Q_n(m) = \vert V_1\vert \,\lambda \left(\prod_{i=2}^m \rho_i(\lambda)\right) (\vert V\vert  + n)^{m-1}
(\lambda \vert V\vert  + n a)^{n-1},
\end{equation}
where $V := \bigcup_{i=1}^m V_i$ and $\rho_i(\lambda) := (1+\lambda)^{\vert V_i\vert }-1$, $i=1,\dots,m$.
\end{theorem}

\begin{proof}
We proceed by induction on $m$. For $m=1$, the contributions to $Q_n(1)$ are given by a number of separate trees, each connected to a vertex of $V_1$. The number of trees on a subset $I_j \subset \{1,\dots,n\}$ of the additional vertices is $\vert I_j\vert^{\vert I_j\vert -2}$, each contributing a factor $a^{\vert I_j\vert -1}$ since $\vert I_j\vert -1$ is the number of edges in the tree, and there is an additional factor $\vert I_j\vert $ due to the choice of vertex connecting to a point of $V_1$. Hence
\begin{eqnarray*}
Q_n(1) &=& \sum_{p=1}^n \sum_{\{I_1,\dots, I_{p}\} \in \Pi_p(\{1,\dots,n\})} (\lambda \vert V_1\vert )^p \prod_{j=1}^p \left(a\vert I_j\vert\right)^{\vert I_j\vert -1} \\ &=& \sum_{p=1}^n \left(\lambda \vert V_1\vert \right)^p {n-1 \choose p-1} (a n)^{n-p} = \lambda \vert V_1\vert  (\lambda \vert V_1\vert  + n
a)^{n-1}.
\end{eqnarray*}
For $m\geq 2$ we define ${\tilde Q}_n(m)$ to be the contribution of forests such that there is no edge between vertices of different sets $V_i$. We claim that
\begin{equation}
\label{Qtilde}
{\tilde Q}_n(m) = {\tilde Q}_n(V_{1};\dots; V_{m}) =\lambda \vert V_1\vert \left(\prod_{i=2}^m \rho_i(\lambda)\right) n^{m-1} (\lambda  \vert V\vert  + n a)^{n-1}.
\end{equation}
Assuming this, we can complete the induction. Indeed, we can subdivide the set $\{1,\dots,m\}$ into subsets
$J_1,\dots,J_q$ such that the $V_i$ with $i \in J_k$ are connected. These forests contribute factors of the form
\begin{equation*}
\vert V_{j_k}\vert \left(\prod_{j \in J_k \setminus \{j_k\}} \rho_j(\lambda)\right) \left( \vert V_{J_k}\vert  \right)^{\vert J_k\vert -2},
\end{equation*}
where $j_k$ is the first vertex in $J_k$ and $j_1=1$. This is analogous to the first application, see Theorem \ref{thm2}. The sets $V_{J_k}$ can now be considered as single sets in a reduced forest (except that an additional vertex is connected to a single $V_i$, i.e. $\rho_{J_k}$ is replaced by $\rho_{j_k}$). Denote by $Q_n(\{J_k\}_{k=1}^q)$ the corresponding contribution. By the claim \eqref{Qtilde}, this contribution equals
\begin{eqnarray*}
Q_n(\{J_k\}_{k=1}^q) &=& \vert V_1\vert  \left(\prod_{j \in J_1 \setminus \{1\}} \rho_j(\lambda)\right) \vert V_{J_1}\vert ^{\vert J_1\vert -2}
\tilde{Q}_n(V_{J_1}; \dots;V_{J_q}) \\ && \times \sum_{j_2 \in J_{2}, \dots, j_q \in J_q} \prod_{k=2}^q \left\{ \vert V_{j_k}\vert  \left(
\prod_{j \in J_k \setminus \{j_k\}} \rho_j(\lambda) \right) \vert V_{J_k}\vert ^{\vert J_k\vert -2} \right\} \\ &=& \lambda
\vert V_1\vert  \left(\prod_{j=2}^m \rho_{j}(\lambda)\right) \left(\prod_{k=1}^q \vert V_{J_k}\vert ^{\vert J_k\vert -1}\right) n^{q-1} (\lambda \vert V\vert  + a n)^{n-1}.
\end{eqnarray*}
Summing over the subdivisions $\{J_k\}_{k=1}^q$ yields
\begin{equation*}
\sum_{\{J_k\}_{k=1}^q} Q_n(\{J_k\}_{k=1}^q) = \lambda \vert V_1\vert  \left(\prod_{j=2}^m \rho_j(\lambda)\right) {m-1 \choose q-1} \vert V\vert ^{m-q} n^{q-1} (\lambda \vert V\vert  + a n)^{n-1}.
\end{equation*}
Identity \eqref{Qnm} then follows by summing over $q=1,\dots,m$. To conclude the proof of Theorem \ref{thm3}, it remains to show \eqref{Qtilde}. The proof essentially relies on the rewriting of ${\tilde Q}_n(m)$ in the form
\begin{eqnarray}
\label{Qtilderesult}
{\tilde Q}_n(m) &=& \lambda \vert V_1\vert \left(\prod_{i=2}^m \rho_i(\lambda)\right)
\sum_{q=1}^{m-1} \sum_{\{J_1,\dots, J_{q}\}\in \Pi_q(\{2,\dots,m\})} \sum_{p=q}^n \sum_{\{I_{1},\dots, I_{p}\} \in \Pi_p(\{1,\dots,n\})} \nonumber
\\ && \times \sum_{(j_k)_{k=1}^{q} \in \mathbb{N}_{p,\neq}^q} \left(\prod_{k=1}^q \vert I_{j_k}\vert ^{\vert J_k\vert }\right) (\lambda \vert V\vert)^{p-1}\prod_{j=1}^p (a
\vert I_j\vert )^{\vert I_j\vert -1}.
\end{eqnarray}
Here, $\mathbb{N}_{p,\neq}^q$ denotes the set of ordered $q$-tuples in $\{1,\dots,p\}$ where each pair is
unequal, i.e.
\begin{equation*}
\mathbb{N}_{p,\neq}^{q} = \left\{(j_k)_{k=1}^q:\,j_k \in \{1,\dots,p\}, j_{k'} \neq j_k \,(k\neq k')\right\}.
\end{equation*}
The expression \eqref{Qtilde} then follows from the following version of the multinomial formula:
\begin{equation*}
\sum_{q=1}^{\min(m-1,n)} \sum_{\{J_1,\dots,J_q\}\in \Pi_q(\{1,\dots,m-1\})} \sum_{(j_k)_{k=1}^{q} \in \mathbb{N}_{p,\neq}^q}  \prod_{k=1}^q x_{j_k}^{\vert J_{k}\vert } = \left(\sum_{j=1}^n x_j \right)^{m-1}.
\end{equation*}
For reader's convenience, the proof of \eqref{Qtilderesult} is postponed to Sec. \ref{proofeq}.
\end{proof}

\section{Proof of Theorem \ref{thm}.}
\label{app}

The proof is done by induction on $m$ and $p$. Note that, for $m=p$ both sides are equal to 1 and for $p=1$ both sides are equal to $\left(\sum_{i=1}^m x_i \right)^{m-1}$. Now assume that identity \eqref{fdeq1} holds true for a given $m \geq 1$ and all $p \leq m$. Let $\mathcal{L}_{m}$ and $\mathcal{R}_{m}$ denote the left-hand side and right-hand side of \eqref{fdeq1} respectively. We may assume that $1 \in I_1$ and expand the factor $\left( \sum_{i \in I_1} x_i \right)^{\vert I_1\vert-1}$ by the binomial formula in powers of $x_1$
\begin{equation*}
\left(\sum_{i \in I_1} x_i \right)^{\vert I_1\vert-1} = \sum_{n=0}^{\vert I_1 \vert-1} {\vert I_1 \vert -1 \choose n} x_1^n \left( \sum_{i \in I_1 \setminus \{1\}} x_i \right)^{\vert I_1 \vert-1-n}.
\end{equation*}
Inserting this into $\mathcal{L}_{m+1}$ and denoting $\tilde{I}_1 := I_1 \setminus \{1\}$, we have
\begin{equation*}
\mathcal{L}_{m+1} =  \sum_{n=0}^{m+1-p} x_1^n \sum_{\substack{\{I_{1},\dots,I_{p}\} \in  \Pi_p(\{1,\dots,m+1\}) \\ 1 \in I_{1},\, \vert I_1\vert \geq n+1}} {\vert \tilde{I}_1 \vert \choose n} \left(\sum_{i \in \tilde{I}_{1}} x_i\right)^{\vert \tilde{I}_{1}\vert - n} \prod_{j=2}^{p} \left(\sum_{i \in I_{j}} x_i \right)^{\vert I_{j}\vert - 1}.
\end{equation*}
If $n=0$, we separate out $\tilde{I}_1=\emptyset$, for which $\{I_2,\dots,I_p\} \in
\Pi_{p-1}(\{2,\dots,m+1\})$. If $\tilde{I}_1 \neq \emptyset$ then $\{\tilde{I}_1,I_2,\dots,I_p\} \in
\Pi_p(\{2,\dots,m+1\})$. Conversely, given a partition $\{\tilde{I}_1,\tilde{I}_2,\dots,\tilde{I}_p\} \in
\Pi_p(\{2,\dots,m+1\})$ we obtain a unique partition of $\{1,\dots,m+1\}$ by adding $1$ to any of the sets
$\tilde{I}_{l}$ with $l \in \{1,\dots,p\}$. We can therefore write
\begin{equation*}
\begin{split}
\mathcal{L}_{m+1} =& \sum_{\{I_{2},\dots,I_{p}\} \in \Pi_{p-1}(\{2,\dots,m+1\})} \prod_{j=2}^{p} \left(\sum_{i \in I_{j}} x_i \right)^{\vert I_{j}\vert - 1}  \\ &+ \sum_{n=0}^{m+1-p} x_1^n \sum_{l=1}^p
\sum_{\substack{\{\tilde{I}_{1},\dots,\tilde{I}_{p} \} \in \Pi_{p}(\{2,\dots,m+1\}) \\ \vert \tilde{I}_{l}\vert \geq n}} {\vert
\tilde{I}_{l} \vert \choose n} \left(\sum_{i \in \tilde{I}_{l}} x_i \right)^{\vert
\tilde{I}_{l}\vert - n} \prod_{\substack{j=1 \\ j\neq l}}^{p} \left( \sum_{i \in \tilde{I}_{j}} x_i \right)^{\vert\tilde{I}_{j}\vert - 1}.
\end{split}
\end{equation*}
Note that in the second term on the right-hand side, $\vert \tilde{I}_{l}\vert \neq 0$ since $\{\tilde{I}_{1},\dots,\tilde{I}_{p} \} \in \Pi_{p}(\{2,\dots,m+1\})$.
Expanding the quantity $\mathcal{R}_{m+1}$ (the right-hand side of \eqref{fdeq1} but with $m+1$) in powers of $x_{1}$ and then replacing $\{2,\dots,m+1\}$ by $\{1,\dots,m\}$, it follows that it suffices to prove the equivalent identities
\begin{multline}
\label{claim0}
\sum_{\{I_{1},\dots,I_{p-1}\} \in \Pi_{p-1}(\{1,\dots,m\})} \prod_{j=1}^{p-1} \left(\sum_{i \in I_{j}} x_i \right)^{\vert I_{j}\vert - 1} \\ + \sum_{l=1}^{p}
\sum_{\{\tilde{I}_{1},\dots,\tilde{I}_{p}\} \in \Pi_p(\{1,\dots,m\})} \left(\sum_{i \in \tilde{I}_{l}} x_i
\right)^{\vert \tilde{I}_{l}\vert} \prod_{\substack{j=1 \\ j\neq l}}^{p} \left( \sum_{i \in \tilde{I}_{j}} x_i \right)^{\vert \tilde{I}_{j}\vert - 1} = {m \choose p-1} \left(\sum_{i=1}^{m} x_i \right)^{m+1-p}
\end{multline}
for $n=0$, and
\begin{multline}
\label{claim}
\sum_{l=1}^p \sum_{\substack{\{\tilde{I}_{1},\dots,\tilde{I}_{p}\} \in \Pi_p(\{1,\dots,m\}) \\ \vert\tilde{I}_{l}\vert \geq
n}} {\vert \tilde{I}_{l} \vert \choose n} \left( \sum_{i \in \tilde{I}_{l}} x_i \right)^{\vert
\tilde{I}_{l}\vert - n} \prod_{\substack{j=1 \\ j\neq l}}^{p} \left( \sum_{i \in \tilde{I}_{j}} x_i \right)^{\vert \tilde{I}_{j}\vert - 1} \\ = {m \choose p-1} {m + 1 - p \choose n} \left(\sum_{i=1}^{m} x_i \right)^{m+1-p-n}
\end{multline}
for all $1 \leq n \leq m+1-p$.
We start with the case of $n=0$. By the induction hypothesis, the first term on the left-hand side of \eqref{claim0} is equal to
\begin{equation*}
{m-1 \choose p-2} \left( \sum_{i=1}^m x_i \right)^{m+1-p}.
\end{equation*}
The second term on the left-hand side of \eqref{claim0} can be rewritten as
\begin{equation*}
\sum_{\{\tilde{I}_{1},\dots,\tilde{I}_{p}\} \in \Pi_p(\{1,\dots,m\})}  \prod_{j=1}^{p} \left( \sum_{i \in \tilde{I}_{j}} x_i \right)^{\vert \tilde{I}_{j}\vert - 1} \sum_{l=1}^p \left(\sum_{i \in \tilde{I}_{l}} x_i \right)= {m-1 \choose p-1} \left( \sum_{i=1}^{m} x_i \right)^{m+1-p},
\end{equation*}
again by the induction hypothesis. Adding the two contributions for $n=0$ yields the right-hand side of \eqref{claim0}. For the case $n=1$, the left-hand side of \eqref{claim} can simply be rewritten as
\begin{equation*}
\begin{split}
\sum_{\{\tilde{I}_{1},\dots,\tilde{I}_{p}\} \in \Pi_p(\{1,\dots,m\})} \left(\sum_{l=1}^p \vert \tilde{I}_{l} \vert\right) \prod_{j=1}^{p} \left(\sum_{i \in \tilde{I}_{j}} x_i \right)^{\vert \tilde{I}_{j}\vert - 1}  &= m {m-1 \choose p-1}  \left( \sum_{i=1}^m x_i \right)^{m-p} \\
&= {m \choose p-1} (m+1-p) \left( \sum_{i=1}^m x_i \right)^{m-p}.
\end{split}
\end{equation*}
Next, we prove the case where $2 \leq n \leq m+1-p$. The key idea is to apply the derivation operator $\sum_{k=1}^{m} \frac{\partial^{n-1}}{\partial x_{k}^{n-1}}$ to the left-hand side of \eqref{fdeq1}. This gives
\begin{multline}
\label{appder}
\sum_{k=1}^{m} \frac{\partial^{n-1}}{\partial x_{k}^{n-1}} \sum_{\{I_{1},\dots,I_{p}\} \in \Pi_p (\{1,\dots,m\})} \prod_{j=1}^{p} \left( \sum_{i \in I_{j}} x_i \right)^{\vert I_{j}\vert-1} \\= \sum_{j'=1}^{p}
\sum_{\{I_{1},\dots,I_{p}\} \in \Pi_p(\{1,\dots,m\})} \prod_{\substack{j=1 \\ j\neq j'}}^{p} \left(\sum_{i \in I_{j}} x_i\right)^{\vert I_{j}\vert-1} \sum_{k \in I_{j'}} \frac{\partial^{n-1}}{\partial x_{k}^{n-1}} \left(\sum_{i\in I_{j'}} x_i \right)^{\vert I_{j'}\vert -1}.
\end{multline}
The $j'$ term on the right-hand side of \eqref{appder} is equal to zero unless $\vert I_{j'}\vert \geq n$. In that case,
\begin{equation*}
\sum_{k \in I_{j'}} \frac{\partial^{n-1}}{\partial x_{k}^{n-1}} \left(\sum_{i\in I_{j'}} x_i \right)^{\vert I_{j'}\vert -1} = \vert I_{j'}\vert \prod_{r=1}^{n-1} (\vert I_{j'}\vert-r) \left( \sum_{i \in I_{j'}} x_i \right)^{\vert I_{j'}\vert-n},
\end{equation*}
independently of $k \in I_{j'}$.  Hence, the left-hand side of \eqref{appder} can be rewritten as
\begin{equation*}
\sum_{j'=1}^p \sum_{\substack{\{I_{1},\dots,I_{p}\} \in \Pi_p(\{1,\dots,m\}) \\ \vert I_{j'}\vert \geq n}} \prod_{r=0}^{n-1} (\vert I_{j'}\vert-r) \left(\sum_{i\in I_{j'}} x_i \right)^{\vert I_{j'}\vert -n} \prod_{\substack{j=1 \\ j\neq j'}}^{p} \left(\sum_{i \in I_{j}} x_i\right)^{\vert I_{j}\vert-1}, \end{equation*}
which is nothing but $n!$ times the left-hand side of \eqref{claim}.  On the other hand, applying the derivation operator $\sum_{k=1}^{m} \frac{\partial^{n-1}}{\partial x_{k}^{n-1}}$ to the right-hand side of \eqref{fdeq1} gives
\begin{equation*}
m {m-1 \choose p-1} \prod_{r=0}^{n-2} (m-p-r) \left( \sum_{i=1}^{m} x_i \right)^{m+1-p-n} = n! {m \choose p-1} {m+1-p \choose n} \left( \sum_{i=1}^{m} x_i \right)^{m+1-p-n},
\end{equation*}
which is just $n!$ times the right-hand side of \eqref{claim}.  This concludes the proof of \eqref{claim}, and hence also the proof of the theorem. \qed

\section{Appendix.}
\label{appex}

\subsection{Proof of \eqref{Qtilderesult}.}
\label{proofeq}

\begin{proof}
We proceed by induction on $m$. For $m=2$, formula \eqref{Qtilde} gives
\begin{equation}
{\tilde Q}_n(2) = \lambda \vert V_1\vert \rho_2(\lambda)
\sum_{p=1}^n \sum_{\{I_1,\dots,I_p\}\in \Pi_p(\{1,\dots,n\})} \sum_{j_1=1}^p \vert I_{j_1}\vert \, (\lambda \vert V\vert)^{p-1} \prod_{j=1}^p (a \vert I_j\vert )^{\vert I_j\vert -1},
\end{equation}
which reduces to \eqref{Qtilde} by Theorem \ref{thm} since $\sum_{j_1=1}^{p} \vert I_{j_1} \vert =n$.\\
For $m \geq 3$, consider first the case that all $V_1,\dots,V_{m-1}$ are already
connected by trees. Denote by ${\tilde Q}_n^{(1)}(m)$ the contribution to ${\tilde Q}_n(m)$ due to forest
graphs where $V_m$ is connected to a single subset $J_l$. We have by induction that this contribution to ${\tilde Q}_n(m)$ is given by
\begin{eqnarray*}
{\tilde Q}_n^{(1)}(m) &=& \lambda \vert V_1\vert \left(\prod_{i=2}^m \rho_i(\lambda)\right) \sum_{q=1}^{m-2}
\sum_{\{J_1,\dots,J_q\} \in \Pi_q(\{2,\dots,m-1\})}  \nonumber \\ && \times \sum_{p=q}^n \sum_{\{I_1,\dots,I_p\} \in \Pi_p(\{1,\dots,n\})} \sum_{(j_k)_{k=1}^q \in \mathbb{N}_{p,\neq}^q} \left(\prod_{k=1}^q \vert I_{j_k}\vert ^{\vert J_k\vert }\right)\left(\lambda \sum_{i-1}^{m-1} \vert V_i \vert\right)^{q-1}\nonumber \\
&& \times \left\{\sum_{l=1}^q  \vert I_{j_l}\vert  (\lambda \vert V\vert) + \sum_{j \notin \{j_1,\dots, j_q\}} \vert I_j\vert  \left(\lambda \sum_{i-1}^{m-1} \vert V_i \vert\right)
\right\} \nonumber \\ && \times (\lambda \vert V \vert)^{p-q-1} \prod_{j=1}^p  (a \vert I_j\vert )^{\vert I_j\vert -1}.
\end{eqnarray*}
Here, the first term in curly brackets corresponds to the case where $V_m$ is connected to the same tree as $J_l$, and the second term corresponds to the case where $V_m$ is connected to $J_l$ for some $l$ by a separate tree. Note that $p-q$ factors in \eqref{Qtilderesult} correspond to trees connected to a single $V_i$.\\
For the first term in curly brackets we set $q'=q$ and define $J'_k = J_k$ ($k=1,\dots,q$) if $k \neq l$ and $J'_l = J_l \cup \{m\}$, and for the second term we put $q'=q+1$ and $J'_k = J_k$ and $J_{q+1} = \{m\}$. Then $\{J'_k\}_{k=1}^{q'}$ is a subdivision of $\{2,\dots,m\}$ and we can write
\begin{eqnarray*}
{\tilde Q}_n^{(1)}(m) &=& \lambda \vert V_1\vert \left(\prod_{i=2}^m \rho_i(\lambda)\right)
\sum_{q'=1}^{m-1} \sum_{\{J_{1}',\dots, J_{q'}'\}\in \Pi_{q'}(\{2,\dots,m\})}  \nonumber \\ && \times \sum_{p=q'}^n \sum_{\{I_1,\dots,I_p\} \in \Pi_p(\{1,\dots,n\})} \sum_{(j_k)_{k=1}^{q'} \in \mathbb{N}_{p,\neq}^{q'}} \prod_{k=1}^q \vert I_{j_k}\vert ^{\vert J'_k\vert } \left(\lambda \sum_{i=1}^{m-1}
\vert V_i \vert\right)^{q'-1} \nonumber \\ && \times (\lambda \vert V \vert)^{p-q'} \prod_{j=1}^p (a \vert I_j\vert )^{\vert I_j\vert -1}.
\end{eqnarray*}
Next, consider the case where there are two groups $S_1, S_2 \subset \{1,\dots,m-1\}$ of $V_i$ ($i \in S_1$ resp. $i \in S_2$) which are connected by trees among themselves but not between each other. Then $S_1$ and $S_2$  have to be connected via $V_m$. A similar configuration is illustrated in Figure \ref{fig2} placed in Sec. \ref{grph}. If $\vert S_1\vert>1$ then $V_m$ can  be connected to $V_{S_1}$ in two ways: directly to one of the trees on $V_{S_1}$ or by a separate tree to one of the vertices of $V_{S_1}$. By induction we therefore have that
\begin{eqnarray*}
{\tilde Q}_n^{(2)}(m) &=& \lambda \vert V_1\vert \left(\prod_{i=2}^m \rho_i(\lambda)\right) \sum_{\{S_1,S_2\} \in \Pi_2(\{1,\dots,m-1\})} \sum_{q_1=1}^{\vert S_1\vert -1} \sum_{\{J_{1}',\dots, J_{q_{1}}'\} \in \Pi_{q_1}(S_1\setminus\{1\})}  \nonumber \\ && \times
\sum_{q_2=1}^{\vert S_2\vert }  \sum_{\{J_{1}'',\dots, J_{q_{2}}''\} \in \Pi_q(S_2)} (\lambda \vert V_m\vert)  (\lambda \vert V_{S_1}\vert )^{q_1-1} (\lambda \vert V_{S_2}\vert )^{q_2-1} \nonumber
\\ &&  \times \sum_{p=q_1+q_2}^n \sum_{\{I_1,\dots,I_p\} \in \Pi_p(\{1,\dots,n\})} \sum_{(j_k)_{k=1}^{q_1+q_2} \in \mathbb{N}_{p,\neq}^{q_1+q_2}} \left(\prod_{k=1}^{q_1} \vert I_{j_k}\vert ^{\vert J'_k\vert } \right) \left(\prod_{k=1}^{q_2} \vert I_{j_{q_1+k}}\vert ^{\vert J''_k\vert }\right) \nonumber \\ &&
\times
\left\{ (\lambda \vert V\vert)  \sum_{l=1}^{q_1} \vert I_{j_l}\vert  + (\lambda \vert V_{S_1}\vert) \sum_{j \notin \{j_1,\dots,j_{q_1+q_2}\}} \vert I_{j}\vert  \right\}
\nonumber \\ && \times (\lambda \vert V\vert)^{p-q_1-q_2-1} \prod_{j=1}^p (a \vert I_j\vert )^{\vert I_j\vert -1}.
\end{eqnarray*}
Note that if $\vert S_1 \vert =1$ then we set $q_1=0$ in which case the sums over $q_1$ and $\{J_1',\dots, J_{q_1}'\}$ are absent as well as the first term in curly brackets.
Now, for the first term in curly brackets, we put $q=q_1+q_2$, $J_k = J'_k$ for $k=1,\dots,q_1$ and $k \neq l$ and $J_l = J'_l \cup \{m\}$, and $J_{k+q_1} = J_k''$ for $k=1,\dots q_2$. This defines a subdivision of $\{2,\dots,m\}$ where the subset containing $m$ has at least two elements. Similarly, for the second term we put $q=q_1+q_2+1$ and $J_k=J_k'$ for $k=1,\dots,q_1$ and $J_{q_1+1} = \{m\}$, and $J_{q_1+1+k} = J_k''$ for $k = 1,\dots,q_2$. This is also a subdivision of $\{2,\dots,m\}$, where in this case $J_{q_1+1}=\{m\}$.
Conversely, suppose that $\{J_k\}_{k=1}^q$ is a subdivision of $\{2,\dots,m\}$. Let $J_l$ be the set containing $m$. Consider first the case where $\vert J_k\vert >1$ for all $k=1,\dots,q$. Then we define the subdivision $\{\tilde{J}_k\}_{k=1}^q$ with $\tilde{J}_k = J_k$ for $k \neq l$ and $\tilde{J}_l = J_l \setminus \{m\}$  and choose an arbitrary subdivision $\{R_1,R_2\}$ of $\{1,\dots,q\}$. We can assume $l \in R_1$ and define $J'_k = \tilde{J}_k$ for $k \in R_1$ and $J''_k = J_{k}$ for $k \in R_2$, and set $S_1= \{1\} \cup \bigcup_{k \in R_1} \tilde{J}_k$ and $S_2 = \bigcup_{k \in R_2} J_k$. The sum over subdivisions $\{R_1,R_2\}$ then reduces to
\begin{equation*}
\sum_{\{R_1,R_2\} \in \Pi_2(\{1,\dots,q\})} (\lambda \vert V_{S_1}\vert)^{q_1-1} (\lambda \vert V_{S_2}\vert )^{q_2-1} = (q-1) \left(\lambda \sum_{i=1}^{m-1} \vert V_i\vert\right)^{q-2},
\end{equation*}
where we used Theorem \ref{thm} with $x_i = \vert V_{\tilde{J}_i}\vert$ for $i \neq l$ and $x_l = \vert V_1\vert + \vert V_{\tilde{J}_l}\vert$.\\
The other extreme is when $\vert J_k\vert =1$ for all $k=1,\dots,q$. In that case we have $q=m-1$ and $J_k = \{k+1\}$. Only the second  term applies and we remove $J_l$ from the list to obtain a subdivision $\{\tilde{J}_k\}_{k=1}^{m-2}$ of $\{2,\dots,m-1\}$, where $\tilde{J}_k = \{k+1\}$. There are now two cases: either $S_1 = \{1\}$ and $J''_k = \{k+1\}$ for $k=1,\dots,m-2=q_2$, or $\vert S_1\vert  > 1$. Again choosing a subdivision $\{R_1,R_2\}$ of $\{1,\dots,m-1\}$, we assume $1 \in R_1$ and set $S_1=\{1\}$ if $R_1=\{1\}$ and $J''_k=\{k+1\}$ for $k=1,\dots,m-2=q_2$, and otherwise $S_1=R_1$ and
$S_2 = \bigcup_{k\in R_2} \{k-1\}$, and $J'_k=\{k\}$ for $k \in R_1$ and $J''_k = \{k-1\}$ for $k \in R_2$. In the latter case $q_1=\vert R_1\vert -1$ and $q_2=\vert R_2\vert $.  The sum over subdivisions $\{R_1,R_2\}$ becomes
\begin{equation*}
\sum_{\{R_1,R_2\} \in \Pi_2(\{1,\dots,q\})} (\lambda \vert V_{S_1}\vert)^{q_1} (\lambda \vert V_{S_2}\vert)^{q_2-1} = (m-2) \left (\lambda \sum_{i=1}^{m-1} \vert V_i\vert \right)^{m-3}.
\end{equation*}
In the general case we can have either $\vert J_l\vert =1$ or $\vert J_l\vert  > 1$. If $\vert J_l\vert  > 1$, this corresponds to the first term in curly brackets and we define the subdivision $\{\tilde{J}_k\}_{k=1}^q$ with $\tilde{J}_k = J_k$ for $k \neq l$ and $\tilde{J}_l = J_l \setminus \{m\}$ and  choose an arbitrary subdivision $\{R_1,R_2\}$ of $\{1,\dots,q\}$. We can assume $l \in R_1$. Then we set $\{J'_k\} = \{\tilde{J}_k\}_{k \in R_1}$ and $\{J''_k\} = \{\tilde{J}_k\}_{k \in R_2}$. If $\vert J_l\vert =1$ then we eliminate $J_l$ and set $\{J'_k\} = \{J_k\}_{k \in R_1 \setminus \{l\}}$ and $\{J''_k\} =
\{J_k\}_{k \in R_2}$. In both cases, the sum over subdivisions $\{R_1,R_2\}$ then reduces to
\begin{equation*}
\sum_{\{R_1,R_2\} \in \Pi_2(\{1,\dots,q\})} (\lambda \vert V_{S_1}\vert)^{q_1-1} (\lambda \vert V_{S_2}\vert)^{q_2-1} = (q-1) \left(\lambda \sum_{i=1}^{m-1} \vert V_i\vert \right)^{q-2},
\end{equation*}
where we used Theorem \ref{thm} with $x_i = \vert V_{\tilde{J}_i}\vert$ for $i \neq l$ and $x_l = \vert V_1 + V_{\tilde{J}_l}\vert$. We conclude that
\begin{eqnarray*}
{\tilde Q}_n^{(2)}(m) &=& \lambda \vert V_1\vert \left(\prod_{i=2}^m \rho_i(\lambda)\right)
\sum_{q=2}^{m-1} (q-1) \sum_{\{J_1,\ldots,J_q\} \in \Pi_{q}(\{2,\dots,m\})}  \nonumber \\ && \times (\lambda \vert V_m \vert) \sum_{p=q}^n \sum_{\{I_1,\dots,I_p\} \in \Pi_p(\{1,\dots,n\})} \sum_{(j_k)_{k=1}^{q} \in \mathbb{N}_{p,\neq}^{q}} \left(\prod_{k=1}^{q} \vert I_{j_k}\vert ^{\vert J'_k\vert }\right)
\nonumber \\ && \times  \left(\lambda \sum_{i=1}^{m-1} \vert V_i\vert \right)^{q-2}
(\lambda \vert V \vert)^{p-q} \prod_{j=1}^p (a \vert I_j\vert )^{\vert I_j\vert -1}.
\end{eqnarray*}
It is now clear that the case where there are more groups $S_1,\dots,S_r$, $r \geq 3$ of subsets connected via $V_m$ is analogous. We again consider the set $J_{l}$ containing $m$ and distinguish the cases where $\vert J_l \vert$ is equal to or greater than 1. We map to subdivisions $\{R_1,\dots,R_r\}$ in the same way as in case $r=2$ and obtain a factor
\begin{equation*}
\sum_{\{R_1,\dots,R_r\} \in \Pi_2(\{1,\dots,q\})} (\lambda \vert V_{S_1}\vert)^{q_1-1} \dots (\lambda \vert V_{S_r}\vert)^{q_2-1} = {q-1 \choose r-1} \left(\lambda \sum_{i=1}^{m-1} \vert V_i\vert\right)^{q-r}.
\end{equation*}
The result is
\begin{eqnarray*}
{\tilde Q}_n^{(r)}(m) &=& \lambda \vert V_1\vert \left(\prod_{i=2}^m \rho_i(\lambda)\right)
\sum_{q=r}^{m-1} {q-1 \choose r-1} \sum_{\{J_1,\dots,J_{q}\} \in \Pi_{q}(\{2,\dots,m\})}  \nonumber \\ && \times (\lambda \vert V_m\vert)^{r-1} \sum_{p=q}^n \sum_{\{I_1,\dots,I_p\} \in \Pi_p(\{1,\dots,n\})} \sum_{(j_k)_{k=1}^{q} \in \mathbb{N}_{p,\neq}^{q}} \left(\prod_{k=1}^{q}
\vert I_{j_k}\vert ^{\vert J'_k\vert }\right) \nonumber \\ && \times  \left(\lambda \sum_{i=1}^{m-1} \vert V_i \vert\right)^{q-r}(\lambda \vert V\vert)^{p-q} \prod_{j=1}^p (a \vert I_j\vert )^{\vert I_j\vert -1}.
\end{eqnarray*}
Finally, summing over $r=1,\dots,m-1$ yields \eqref{Qtilderesult}.
\end{proof}

\subsection{Some illustrations.}
\label{grph}

\newpage
\begin{figure}[h!]
\centering
\includegraphics[width=12cm,height=12cm,keepaspectratio]{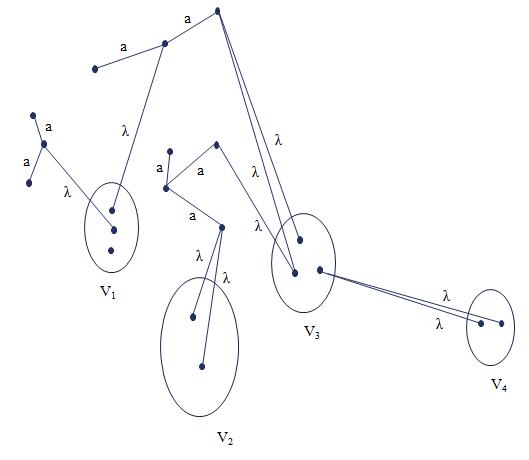}
\caption{A configuration with 4 sets of vertices and 10 additional single vertices.}
\label{fig1}
\end{figure}

\begin{figure}[h!]
\centering
\includegraphics[width=12cm,height=12cm,keepaspectratio]{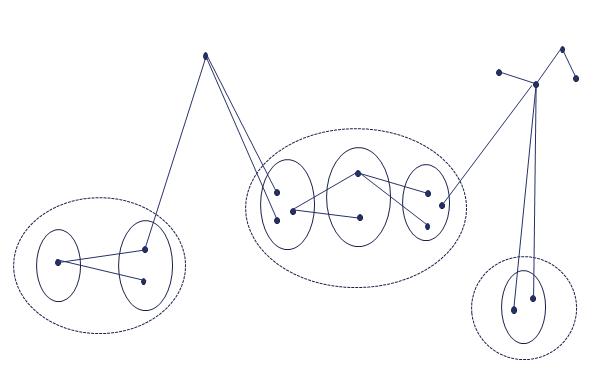}
\caption{2 groups of set of vertices and 1 single set of vertices connected to each other by trees only (i.e., not directly connected to each other).}
\label{fig2}
\end{figure}

\subsection*{Acknowledgment.}
The authors thank Adrien Kassel for having brought to their attention reference \cite{Stan}. The authors also thank Benjamin Hackl for having brought to their attention his paper \cite{HB} related to our main result Theorem \ref{thm}.
The second author gratefully acknowledges the financial support of the Ukrainian Scientific Project "III-12-16 Research of models of mathematical physics describing deterministic and stochastic processes in complex systems of natural science".

\small{
}

\end{document}